\documentclass{amsproc}

\usepackage{amsfonts, amsmath, amsthm, amssymb, latexsym}

\newtheorem{theorem}{Theorem}[section]
\newtheorem{lemma}[theorem]{Lemma}
\newtheorem{corollary}[theorem]{Corollary}

\theoremstyle{definition}

\newtheorem{example}[theorem]{Example}

\theoremstyle{remark}
\newtheorem{remark}[theorem]{Remark}

\numberwithin{equation}{section}

\begin{document}

\title[Martingales for Heckman-Opdam processes]{Elementary symmetric polynomials and martingales for Heckman-Opdam processes}


\author{Margit R\"osler}
\address{Institut f\"ur Mathematik, Universit\"at Paderborn, Warburger Str. 100, D-33102 Paderborn, Germany}
\email{roesler@math.upb.de}

\author{Michael Voit}
\address{Fakult\"at Mathematik, Technische Universit\"at Dortmund, Vogelpothsweg 87, D-44221 Dortmund, Germany}
\email{michael.voit@math.tu-dortmund.de}


\date{}

\begin{abstract}
We consider the generators $L_k$ of Heckman-Opdam diffusion processes in the compact and non-compact case in $N$ dimensions
 for root systems of type $A$ and $B$, with a multiplicity function of the form $k=\kappa k_0$ with some fixed value $k_0$
 and a varying constant $\kappa\in\,[0,\infty[$.
     Using  elementary symmetric functions, we  present polynomials which are simultaneous eigenfunctions
     of the  $L_k$
      for all $\kappa\in\,]0,\infty[$. This  leads to
      martingales associated with the  Heckman-Opdam diffusions $ (X_{t,1},\ldots,X_{t,N})_{t\ge0}$.
     As our results extend to the freezing case $\kappa=\infty$  with a deterministic
     limit after some renormalization, we find formulas  for the expectations
     $\mathbb E(\prod_{j=1}^N(y-X_{t,j})),$ $y\in\mathbb C$.

\end{abstract}

\maketitle

\section{Introduction}

In the theory of classical random matrix ensembles 
there exist several formulas regarding determinants as follows: 
Let $X$ be a random variable with values in some space of $N\times N$-matrices over 
$\mathbb F=\mathbb R, \mathbb C,\mathbb H $. Then
the expectations $\mathbb E(\det(X-yI_N))$ are  
 classical orthogonal polynomials of degree $N$ in $y\in\mathbb C$.
Such results for the
Hermite, Laguerre, and Jacobi ensembles can for example be found  in \cite{DG, FG, A},
where the expectations above can be expressed via classical Hermite, Laguerre and Jacobi polynomials of degree $N$.
These results were extended to the spectra of $\beta$-ensembles
associated with $N$-dimensional time-homogeneous diffusion processes
in \cite{KVW, V} by using some martingales constructed from  these  diffusions via elementary symmetric polynomials. The
generators of the diffusions were Dunkl-Bessel Laplacians and (symmetric) Heckman-Opdam Laplacians in the compact $BC$ case.

In the present paper we study corresponding results for 
Heckman-Opdam Laplacians in the non-compact  $BC$ case as well 
Heckman-Opdam Laplacians of type $A$ in the compact and noncompact setting.
Together with \cite{KVW, V}, the present paper covers the most important  examples related with Calogero-Moser-Sutherland particle models.
The basic ideas here are similar to  those in \cite{KVW, V};
however, while the approach in \cite{KVW, V}
is  mainly based on Ito calculus, we now focus on an algebraic point of view.

The idea is  as follows. Let $(X_t)_{t\ge0}$ be a time-homogeneous diffusion  on a suitable 
closed set $C\subset \mathbb R^N$ (such as a Weyl chamber or a fundamental alcove), where the paths  are reflected at the boundary $\partial C$.
Then the generator $L$ of the associated transition semigroup is  an elliptic partial differential operator
whose domain is contained in the space  of 
 functions  on $\mathbb R^N$ which admit corresponding symmetries on $\partial C$.
  We are now interested in functions $f:[0,\infty[\times\mathbb R^N\to\mathbb C$ for which $(f(t,X_t))_{t\ge0}$ is a martingale
     (w.r.t.~the canonical filtration),
    which essentially means that $f$ is $L$-space-time-harmonic, i.e., $(\frac{\partial}{\partial t}+L)f=0$; see Section III.10 of \cite{RW}.
     Examples of such harmonic functions can be given in terms of eigenfunctions of $L$. For a general background in stochastic analysis we also recommend \cite{P}.

     In the framework of Heckman-Opdam theory (see \cite{HS}, \cite{HO}),
     we fix some crystallographic
     root system $R$ (with associated Weyl group $W$)  in $\mathbb R^N$ and choose $C$  as an associated 
Weyl chamber or  fundamental alcove. We consider a non-negative
multiplicity function $k\ge0$ on $R$ of the form $k=\kappa \cdot k_0$
with some fixed multiplicity  $k_0$ and a constant $\kappa>0$. We  now study the $W$-invariant Heckman-Opdam Laplace operators
 $L_\kappa:= L_{\kappa k_0}$ as generators of diffusions (see \cite{Sch1, Sch2, RR1}), 
  where  the parameter $\kappa$ is varying. 
We also study the renormalized generators 
$\widetilde L_\kappa:=\frac{1}{\kappa}L_\kappa$ of the renormalized diffusions  $(\widetilde X_t:= X_{t/\kappa} )_{t\ge0}$.
In suitable coordinates, these  renormalized  generators
then have the form $$\widetilde L_\kappa f=\frac{1}{\kappa}\Delta f + Hf$$ with
some second-order differential operator $\Delta$  (often a classical Laplacian)  and some first-order drift operator $H$,  where both operators are
independent of $\kappa$. This  also works for $\kappa=\infty$ where $(\widetilde X_t )_{t\ge0}$
is  the deterministic  solution of some  ODE associated with $H$.
We now use elementary symmetric polynomials to construct simultanous
eigenfunctions $f$ of $\widetilde L_\kappa$ for all $\kappa\in]0,\infty]$.
In the next step we use  martingales associated with these functions $f$  together with information in the  deterministic case
$\kappa=\infty$ in order to
derive formulas for 
\begin{equation}\label{eq-intro}
  \mathbb E(\prod_{j=1}^N(\widetilde X_{t,i}-y))  \quad\quad(y\in\mathbb C).
  \end{equation}
For some values of $\kappa$,  the  operators $L_\kappa$ 
admit an interpretation as Laplace-Beltrami operators on symmetric spaces.  In these cases, 
 (\ref{eq-intro})
leads to determinantal formulas for Brownian motions on the
symmetric spaces which are closely related, for instance, to some identities in \cite{R}.
Considering $t\to \infty$ in the compact cases, this also leads to determinantal formulas for the uniform distributions
on compact symmetric spaces.

\medskip
The paper is organized as follows: In Section 2 we first recapitulate some  well-known
facts on  Heckman-Opdam hypergeometric functions and polynomials, and the associated Laplacians in the compact and non-compact setting.
We there also prove that in the non-compact  crystallographic case, the Heckman-Opdam processes
admit
arbitrary exponential moments for arbitrary deterministic starting conditions.
In the main part of the paper,  we then concentrate on specific root systems.
Section 3 is devoted to the compact case of type $A$, which is related to Calogero-Moser-Sutherland particle models on the torus.
Here, for instance, our results lead to  determinantal formulas for Brownian motions and the uniform distribution on the unitary groups $U(N)$
and  $SU(N)$. 
In Section 4 we  study the  non-compact case of type $A$.
Section 5 then contains the   non-compact case of type $BC$.

\section{Heckman-Opdam theory}

Here we  briefly collect some  facts from  Heckman-Opdam theory; see  \cite{HS, HO} for a general background.

 Let  $(\frak a, \langle\,.\,,\,.\,\rangle) $ be a Euclidean space of dimension $N$ with norm $|x|= \sqrt{\langle x, x\rangle}.$ We identify $\frak a$ with its dual space via the given scalar product.
 Let $R$ be a crystallographic, possibly not reduced root system in $\frak a$
 with associated finite reflection group $W$. Thus in particular, $R$ spans $\frak a.$ 
 We fix a positive subsystem $R_+\subset R$ and a $W$-invariant multiplicity  $k:R\to[0,\infty[$. 
 The Cherednik operators associated with $R_+$ and $k$  are defined as
 \begin{equation}\label{def-dunkl-cher-op} 
D_\xi(k)f(x) 
= \partial_\xi f(x) + \sum_{\alpha \in R_+}  \frac{ k_\alpha \langle \alpha, \xi\rangle }{1-e^{-\langle \alpha,x\rangle}}
 (f(x)-f(\sigma_\alpha(x)) -\langle \rho(k), \xi\rangle f(x) 
\end{equation} 
for $\xi \in \frak a,$ with the (weighted) Weyl vector
 $ \rho(k) := \frac{1}{2} \sum_{\alpha \in R_+} k_\alpha\alpha$.
The operators $D_\xi(k), \, \xi \in \frak a\,$  commute, and there is a $W$-invariant tubular neighbourhood $U$ of $\frak a$ in $ \frak a_{\mathbb C} = \frak a + i\frak a$ and
a unique analytic function $(\lambda, z) \mapsto G(\lambda,k;z)$ on 
 $\frak a_{\mathbb C} \times U$, the so called Opdam-Cherednik kernel,  which satisfies
 \begin{equation}\label{Opdam-Cherednik}
   G(\lambda,k; 0) = 1 \quad\text{ and} \quad D_\xi(k)G(\lambda, k;\,.\,) =
   \langle \lambda, \xi \rangle\, G(\lambda,k; \,.\,) \quad\text{ for all }\,\xi \in \frak a.
\end{equation} 
 The hypergeometric function associated with $R$  is  defined by
 \begin{equation}\label{hypergeom-fct} F(\lambda,k;z) = \frac{1}{|W|} \sum_{w\in W} G(\lambda,k; w^{-1}z).
   \end{equation} 
It is $W$-invariant in $\lambda$ and $z.$ To introduce the associated Heckman-Opdam polynomials,
we need the weight lattice and the set of dominant weights,  
$$ P = \{\lambda \in \frak a: \langle \lambda, \alpha ^\vee \rangle \in \mathbb Z \>\> \forall \alpha \in R\,\}, \quad
P_+ = \{\lambda \in P: \langle \lambda, \alpha^\vee\rangle \geq 0\,\, \forall \alpha \in R_+\,\}\supset R_+ ,$$ 
where $\alpha^\vee = \frac{2\alpha}{\langle \alpha, \alpha \rangle}.$ 
 $P_+$ carries the usual dominance order. 
Let
$$ \mathcal T:= \text{span}_{\mathbb C}\{e^{i\lambda}, \, \lambda \in P\}$$
be the space of trigonometric polynomials associated with $R$.
The orbit sums 
$$ M_\lambda = \sum_{\mu \in W\!\lambda} e^{i\mu}\,, \quad \lambda \in P_+$$
form a basis of the subspace $\mathcal T^W$ of $W$-invariant polynomials in $\mathcal T$.
For $ Q^\vee := \text{span}_{\mathbb Z}\{\alpha^\vee, \, \alpha \in R\}$,
 consider the compact torus $\,T = \frak a/2\pi Q^\vee$ with 
the weight function 
\begin{equation}\label{weight-trig} \delta_k(x) := 
\prod_{\alpha \in R_+} \Bigl|\sin\Bigr(\frac{\langle \alpha, x\rangle}{2}\Bigr)\Bigr|^{2k_\alpha}.
\end{equation}
The Heckman-Opdam polynomials associated with $R_+$ and $k$  are given by
$$ P_\lambda(k;z) = M_\lambda(z) + \sum_{\nu < \lambda} c_{\lambda\nu}(k) M_\nu(z) \quad (\lambda \in P_+\,, z\in \frak a_{\mathbb C})$$
where the  coefficients $c_{\lambda\nu}(k)\in \mathbb R$ are  uniquely determined  by the condition that $P_\lambda(k;\,.\,)$
 is orthogonal to $M_\nu$ in $L^2(T, \delta_k)$ for all $\nu \in P_+$ with $\nu<\lambda$.
 It is known that 
 $\{P_\lambda(k, \,.\,), \lambda\in P_+\,\}$ is an orthonormal basis of  $L^2(T, \delta_k)^W$
 of all $W$-invariant functions from $L^2(T, \delta_k)$. 
According to \cite{HS}, the normalized polynomials
$$ R_\lambda(k,z):= P_\lambda(k;z)/P_\lambda(k;0)$$
can be expressed in terms of the hypergeometic function as
\begin{equation}\label{relation-pol-hyper}
 R_\lambda( k,z) = F(\lambda+\rho(k),k;iz).
\end{equation}
Note that our notion slightly differs from  \cite{HS, HO}, where the polynomials $P_\lambda$ are defined as 
exponential polynomials on the torus $i\frak a /2\pi i Q^\vee$.  
We next introduce the Heckman-Opdam Laplacian
$$
  \Delta_k:= \sum_{j=1}^N D_{\xi_j}(k)^2 \,- \, |\rho(k)|^2$$
with some orthonormal basis $\xi_1,\ldots,\xi_N$ of $\frak a$.
The operator $\Delta_k$ is independent of the choice of this basis. Denote by $L_k$ the restriction of  $\Delta_k$ to $W$-invariant functions. According to \cite{Sch2}, 
\begin{equation}\label{HO-Lap} L_k f(x)= \Delta f(x)+\sum_{\alpha \in R_+} k_\alpha \coth\Bigl( \frac{\langle \alpha, x\rangle}{2}\Bigr)\cdot
 \partial_\alpha f(x).\end{equation}
 We notice that by construction, for all $\lambda\in\frak a_{\mathbb C}$, the hypergeometric functions 
$$F_\lambda := F(\lambda,k;\,.\,)$$ 
 are eigenfunctions of $L_k$ with eigenvalues $$  \sum_{j=1}^N\langle \lambda, \xi_j \rangle^2-|\rho(k)|^2.$$
The operator $L_k$ is independent of the choice of $R_+$ and  generates a Feller diffusion on the closed 
 Weyl chamber $\overline{\frak a_+}\subset \frak a$ associated with $R_+$ (called a radial Heckman-Opdam process), where
 the paths are reflected at the boundaries, see \cite{Sch1, Sch2}. The transition probabilities of this process, with starting point $y \in \overline{\frak a_+}$, are given by
 $$ p_t^W\!(x,y)d\mu(x) $$
 with the $W$-invariant heat kernel
 $$ p_t^W\!(x,y) = \int_{i\frak a} e^{-\frac{1}{2}(|\lambda|^2 + |\rho|^2)} F_\lambda(x) F_\lambda(-y) d\nu^\prime(\lambda),$$ 
 where $\rho=\rho(k),$ $\nu^\prime$ is the symmetric Plancherel measure as in  \cite{Sch2}, and 
 $$ d\mu(x)= d\mu_k(x) = \prod_{\alpha \in R_+} \big\vert 2 \sinh \langle \frac{\alpha}{2}, x \rangle\big\vert^{2k_\alpha}\, dx.$$
 
 In the compact case  we take the factor $i$ in (\ref{relation-pol-hyper}) into account as in \cite{RR1} and 
  consider the operator
\begin{equation}\label{generator-general}
\widehat L_kf(x):=\Delta f(x)+\sum_{\alpha \in R_+} k_\alpha\cot\Bigl( \frac{\langle \alpha, x\rangle}{2}\Bigr)\cdot
 \partial_\alpha f(x).
\end{equation}
$\widehat L_k$ generates a Feller diffusion on a compact fundamental alcove of the affine Weyl group $W_{\!\mathit{aff}}=W \ltimes 2\pi Q^\vee$ in $\frak a$, again 
with reflecting boundaries.
The trigonometric polynomials $P_\lambda$ ($\lambda\in P_+$) 
are eigenfunctions of $\widehat L_k$ 
with  eigenvalues $$-\langle \lambda,\lambda+2\rho(k)\rangle\le0.$$
They were used in \cite{RR1} to construct the transition densities of the diffusions with generators $\widehat L_k$.
For  root systems of type $BC$, the $P_\lambda$ are multivariate Jacobi polynomials, and the associated
Feller diffusions are  multivariate Jacobi processes which were studied in \cite{Dem}.

Before turning to details for root systems of type $A$ and $BC$, we conclude this section with an integrability result which assures the existence of exponential moments for the radial Heckman-Opdam processes in the non-compact, crystallographic case.  This will be needed in Sections \ref{Sec_4} and \ref{Sec_5}. 

\begin{lemma}\label{exponential-moments-ex} For each $y \in \frak a$ and $\,\beta\in \frak a,$ 
\begin{equation}\label{moments} \int_{\frak a} e^{\langle \beta, x\rangle} \,
p_t^W\!(x,y) d\mu(x) \, < \infty.\end{equation}.
\end{lemma}

\begin{proof} For $y=0$ this is obvious from Theorem 5.2 of \cite{Sch2}. For general $y$, we employ the $L^p$-theory for the hypergeometric transform developed in \cite{NPP}. 
For suitable functions on $\frak a$ and on $i\frak a$ respectively, the hypergeometric transform and the inverse hypergeometric transform are given by
$$ \mathcal Hf(\lambda) = \int_{\frak a} f(x) F_\lambda(-x) d\mu(x), \quad 
\mathcal I(g)(x) = \int_{i\frak a} g(\lambda) F_\lambda(x) d\nu^\prime(\lambda).	$$
We denote by $\mathbb C[\frak a_{\mathbb C}]$ the space of polynomial functions on $\frak a_{\mathbb C}$ and by $\partial (q)$ the constant coefficient differential operator associated with  $q\in \mathbb C[\frak a_{\mathbb C}].$ Moreover, for $x \in \frak a$ we denote by $C(x):= co(W.x)$ the convex hull of the $W$-orbit of $x$ in $\frak a.$ 
	For an exponent $0<p\leq 2$, the $W$-invariant $L^p$-Schwartz space is given by
	$$ \mathcal C^p(\frak a)^W = \Big\{f \in C^\infty(\frak a)^W: \, \sup_{x\in \frak a}\, \frac{(1+|x|)^n}{F_0(x)^{2/p}}\, \big\vert\partial(q) f(x)\big\vert \, < \infty \,\, \forall \,  n \in \mathbb N_0, \, q \in \mathbb C[\frak a_{\mathbb C}]\Big\}.$$
	Moreover, we consider the $W$-invariant Schwartz space $ \mathcal S(\frak a_{\epsilon_p})^W$ with $\epsilon_p = \frac{2}{p}-1,$ 
	which consists of all $W$-invariant continuous functions on the closed	
	 tube $\frak a_{\epsilon_p} = C(\epsilon_p\rho) + i \frak a \subset \frak a _{\mathbb C}$ which are holomorphic in its interior and satisfy
	 	 \begin{equation}\label{growth} \sup_{\lambda \in \frak a_{\epsilon_p}}\,(1+|\lambda|)^n \big\vert\partial(q)g(\lambda)\big\vert \, < \infty \quad \text{for all } n\in \mathbb N_0, \, q \in \mathbb C[\frak a_{\mathbb C}]. \end{equation}
	 By Theorem 5.6 of \cite{NPP}, the hypergeometric transform
	 $ \mathcal H$ 
	 is a topological isomorphism from $\mathcal C^p(\frak a)^W$ onto $	\mathcal S(\frak a_{\epsilon_p})^W$	 with inverse $\mathcal I$. 
	 
	 \smallskip\noindent
	 We claim that for fixed $y \in \frak a,$ the function 
	 $$ g(\lambda):= e^{\frac{1}{2}\langle \lambda,\lambda\rangle} F_\lambda(-y)$$
	 belongs to  $\mathcal S(\frak a_{\epsilon_p})$ for each $p\in ]0,2]; $ here $\langle \,.\,,\,.\,\rangle $ denotes the bilinear extension of the given scalar product to $\frak a_{\mathbb C}.$ As soon as this is proven, it will follow that $\,p_t^W(\,.\,,y) = e^{-\frac{1}{2}|\rho|^2}\,\mathcal I (g) \,$ belongs  to $ \mathcal C^p(\frak a)^W$ for each $p \in ]0,2].$
	 In order to check that $g\in \mathcal S(\frak a_{\epsilon_p}),$	 we only have to verify the growth condition \eqref{growth}. Let $q\in \mathbb C[\frak a_{\mathbb C}].$ Then in view of Theorem 3.4 (and Remark 3.2) of \cite{Sch2}, 
	 $$ |\partial_\lambda(q) F_\lambda(-y)| \, \leq C (1+|y|)^{\deg q} F_0(-y)\cdot e^{\max_{w\in W} \text{Re}\langle w\lambda, -y\rangle} \quad (\lambda \in \frak a_{\mathbb C}).$$
	 This shows that $\partial_\lambda(q) F_\lambda(-y)$ is bounded as a function of $\lambda$  on $\frak a_{\epsilon_p}$.  
	 Moreover,  
	 $\partial(q) e^{\frac{1}{2} \langle \lambda, \lambda\rangle} = \widetilde q(\lambda)e^{\frac{1}{2} \langle \lambda, \lambda\rangle}$ with some polynomial $\widetilde q$, and 
	 $$ \big\vert e^{\frac{1}{2} \langle \lambda, \lambda\rangle}	\big\vert \, \asymp\,  e^{-\frac{1}{2}|\lambda|^2} \quad \text{on }   \frak a_{\epsilon_p}.$$	 
	 Therefore $\partial(q)g(\lambda)$ decays exponentially as $|\lambda|\to \infty $ within $\frak a_{\epsilon_p}.$ It follows that $g \in \mathcal S(\frak a_{\epsilon_p})$ and thus $\,p_t^W\!(\,.\,,y) \in \mathcal C^p(\frak a)^W.$	 In particular, for each $p \in ]0,2]$ there exists a constant $C_p >0$ such that  
	  $$ p_t^W\!(x,y) \leq C_p  F_0(x)^{2/p} \quad  \text{ for all } x\in \frak a.$$
	  From \cite{Sch2} we know that in the closed chamber $\overline{\frak a_+}\,,$ $\, F_0(x) \asymp q_0(x) e^{-\langle \rho,x\rangle}\,$  
	 with a certain positive polynomial $q_0$. Hence there exists a nonnegative polynoimal $q_p$ (depending on $p$), such that
	 $$ p_t^W\!(x,y) \leq q_p(x) e^{-\frac{2}{p}\langle\rho,x\rangle} \quad \text{for all } x\in \overline{\frak a_+}.$$	 
	  Now fix $\beta \in \frak a$ and note that $\rho$ is contained in the open chamber $\frak a_+.$ Choosing $p>0$ small enough, we therefore obtain 
	 $$ \int_{\overline{\frak a_+}} e^{\langle \beta, x\rangle} p_t^W\!(x,y)d\mu(x) \, \leq \int_{\overline{\frak a_+}}  q_p(x)e^{\langle \beta, x\rangle  -\frac{2}{p}\langle \rho, x \rangle  + 2 \langle \rho, x\rangle} dx\, 
	 < \infty. $$
	 This yields the assertion.
		 	 \end{proof}


\section{The compact  case of type $A_{N-1}$}

In this section we study Heckman-Opdam processes of type $A_{N-1}$ in the compact setting. The generators of these processes are the Hamiltonians of
interacting particle  models of Calogero-Sutherland type
with $N$  particles on the torus $\mathbb T:=\{z\in\mathbb C:\> |z|=1\}$; see \cite{LV} for the background.
These processes are diffusions
 on some fundamental domain of $W=S_N$ in  $\mathbb T^N$.
 It will however be convenient  to consider also associated diffusions 
 on  $\mathbb R^N$ with  $2\pi$-periodicity such that the diffusions on $\mathbb T^N$ appear as
images under $x \mapsto e^{ix}.$ 
To introduce the processes on $\mathbb R^N$, we consider the root system $R=A_{N-1}=\{\pm (e_i-e_j): 1 \leq i < j \leq N\}$ in $\mathbb R^N$ with positive subsystem $R_+ = \{e_i-e_j: i<j\}$ and fix a multiplicity parameter $k\in\,]0,\infty[$. Let 
$$ \omega:= (1, \ldots, 1)^T \in \mathbb R^N.$$
Then $Q^\vee = \mathbb Z^N \cap (\mathbb R \omega)^\perp,$ and a fundamental domain for the action of $W_{\!\mathit{aff}}= W \ltimes 2\pi Q^\vee$ in $\mathbb R^N$ is given by
\begin{align*} C_N &= \{ x\in \mathbb R^N: 0 \leq \langle \alpha, x \rangle \leq 2\pi \, \, \forall \alpha \in R_+\}\\
  & = \, \{x\in \mathbb R^N: \, x_1\le x_2\le\ldots\le x_N\le x_1+2\pi\}.
  \end{align*}
 We consider the $W$-invariant Heckman-Opdam
  Laplacian
  \begin{equation}
    \widehat L_kf(x)= \Delta f(x)+ k\sum_{j=1}^N \sum_{l\ne j}
    \cot\Bigl(\frac{x_j-x_l}{2}\Bigr)\frac{\partial}{\partial x_j}f(x)
   \end{equation}
  with reflecting boundaries, i.e. with domain
  $$ D(\widehat L_k) = \{f\vert_{C^N}: f \in C^2(\mathbb R^N)\text{ invariant under }W_{\!\mathit{aff}} \}.$$ 

$\widehat L_k$ is the generator of a Feller semigroup of transition operators on $C_N$, c.f. \cite{RR1}. 
Associated Feller diffusions $(X_{t,k})_{t\ge0}$ with
continuous paths (which are reflected at the boundary of $C_N$) are called Heckman-Opdam processes of type $A_{N-1}$ on $C_N$.
We  also consider the renormalized generators
$$\widetilde L_k:=\frac{1}{k} \widehat L_k= \frac{1}{k}\Delta+ \sum_{j=1}^N \sum_{l\ne j}
 \cot\Bigl(\frac{x_j-x_l}{2}\Bigr)\frac{\partial}{\partial x_j}$$
which degenerate for $k\to\infty$ into 
$$\widetilde L_\infty=\sum_{j=1}^N \sum_{l\ne j}
 \cot\Bigl(\frac{x_j-x_l}{2}\Bigr)\frac{\partial}{\partial x_j}.$$
For $k\in]0,\infty[$, the process  $\widetilde X_{k}:=(\widetilde X_{t,k})_{t\geq 0}$ with $ \widetilde X_{t,k}:= X_{t/k,k}$ is a Feller diffusion associated with
     $\widetilde L_k$. It can be also described as solution of the SDE
     \begin{equation}\label{sde-compact-a}
       d\widetilde X_{t,k,j}=\frac{\sqrt 2}{\sqrt k}dB_{t,j}+\sum_{l\ne j} \cot\Bigl(\frac{\widetilde X_{t,k,j}-\widetilde X_{t,k,l}}{2}\Bigr)dt
       \quad\quad(j=1,\ldots,N)
     \end{equation}
    with some $N$-dimensional Brownian motion $(B_{t,1},\ldots,B_{t,N})_{t\ge0}$.
    Moreover, for deterministic starting conditions, 
  $\widetilde L_\infty$ is the generator of a deterministic process whose paths are the solution of some
  initial value problem for the ODE 
  \begin{equation}\label{dgl-compact-a}
    \frac{dx_j}{dt}(t)= \sum_{l\ne j}\cot\Bigl(\frac{x_j(t)-x_l(t)}{2}\Bigr) \quad\quad(j=1,\ldots,N).
    \end{equation}
  As for Dunkl processes in \cite{AV}, one can show that for initial data in the interior of $C_N$,
  the solution $(\widetilde X_{t,\infty})_{t\ge0}$ of this ODE exists for all $t\ge0$ in the interior of $C_N$. We also point out that in the Dunkl setting, 
  the ODE analogous to \eqref{sde-compact-a} has unique solutions for all $t\geq 0$ 
 even for starting points at the boundary of the chamber, see \cite{VW}. We expect that such a result is also true in the present setting.
  
  \medskip
  
   We need the following stationary solutions of 
  (\ref{dgl-compact-a}):

  \begin{lemma}\label{stat-solution-a} For each $x_1\in\mathbb R$,
    $$\Bigl(x_1,x_1+\frac{1}{N}2\pi,x_1+\frac{2}{N}2\pi,\ldots, x_1+\frac{N-1}{N}2\pi\Bigr)\in C_N$$
    is a stationary solution of  (\ref{dgl-compact-a}).
\end{lemma}

  \begin{proof} Assume first that $N$ is odd. As the cotangent is odd and $\pi$-periodic, we have for $j=1,\ldots,N$ that
    $$\sum_{l\ne j} \cot\Bigl(\frac{(j-l)\pi}{N}\Bigr)=\sum_{l=1}^{N-1}\cot\Bigl(\frac{l\pi}{N}\Bigr)
    =\sum_{l=1}^{(N-1)/2}\Bigl(\cot\Bigl(\frac{l\pi}{N}\Bigr)+\cot\Bigl(\frac{(N-l)\pi}{N}\Bigr) \Bigr)=0.$$
    If $N$ is even, then our computation leads to the additional term $\cot((N\pi/2)/N)=0$ in the last sum and
    thus to the same result. 
    This yields the claim.
  \end{proof}

  The generator $\widetilde L_k$ and the associated diffusion $\widetilde X_k$ on $\mathbb R^N$ can be decomposed into two independent parts,
  namely the  center of gravity and the process of the distances of neighboring particles.
  This reflects the fact that the usual representation of the symmetric group $S_N$ on $\mathbb R^N$
  decomposes into two irreducible components.  More precisely, consider the  center-of-gravity process
  $\widetilde X_{k}^{\mathit{cg}}:=(\widetilde X_{t,k}^{\mathit{cg}})_{t\ge0}$ with
  $$\widetilde X_{t,k}^{cg}:= \frac{1}{N}(\widetilde X_{t,k,1}+\ldots+\widetilde X_{t,k,N})\cdot \omega$$
  which is the orthogonal projection of $\widetilde X_{t,k}$ onto $\mathbb R\omega$. Then the
  diffusion 
  $$\widetilde X_{k}^{\mathit{diff}}:=\widetilde X_{k}-\widetilde X_{k}^{\mathit{cg}}$$
  lives  on the orthogonal complement $(\mathbb R\omega)^\perp\subset \mathbb R^N$. 

  \begin{lemma}\label{independence-a}
    Let $k\in\,]0,\infty[$. If $\widetilde X_k$ starts in  some deterministic point, then the processes $\widetilde X_{k}^{\mathit{diff}}$ and
 $\widetilde X_{k}^{\mathit{cg}}$ are stochastically independent.
\end{lemma}
  
  \begin{proof} By the SDE (\ref{sde-compact-a}) we have
    \begin{equation}\label{sde-diagonal-compact-a}
      d\widetilde X_{t,k,j}^{\mathit{cg}} = \frac{\sqrt 2}{N\sqrt k}d\Bigl( \sum_{l=1}^N B_{t,l}\Bigr)=:\frac{\sqrt 2}{\sqrt{Nk}}\, dB_t
      \end{equation}
    with some one-dimensional Brownian motion $(B_t)_{t\ge0}$
    while  $\widetilde X_{k}^{\mathit{diff}}$ satisfies
    $$d\widetilde X_{t,k,j}^{\mathit{diff}} =\frac{\sqrt 2}{\sqrt k}\,d\Bigl(B_{t,j}-\frac{1}{N}\sum_{l=1}^NB_{t,l}\Bigr)
    +F_j(\widetilde X_{t,k}^{\mathit{diff}})\,dt \quad(j=1,\ldots,N)$$
    where the processes $(B_{t,j}-\frac{1}{N}\sum_{l=1}^NB_{t,l})_{t\ge0}\,$ are stochastically independent of $(B_t)_{t\ge0}$, and the 
$F_j$ are concrete continuous functions. This implies the claim by the very definition of solutions of SDEs.
  \end{proof}

  We next transfer all data from $\mathbb R^N$ to the torus $\mathbb T^N$ via $z_j:=e^{ix_j}$ for $j=1,\ldots,N$, or for short,
  $z:=e^{ix}\in \mathbb T^N$. A short computation shows that in $z$-coordinates, the operator $L_k$ is given by the diffusion operator
    \begin{align}\label{Lap-Vinet-op}
      H_k =&  -\sum_{j=1}^N\Bigl( z_j\frac{\partial}{\partial z_j}\Bigr)^2-
      k\sum_{j=1}^N \sum_{l\ne j}\frac{z_j+z_l}{z_j-z_l}\cdot z_j\frac{\partial}{\partial z_j}\\
      =& -\sum_{j=1}^N z_j^2 \frac{\partial^2}{\partial z_j^2} -(1-k(N-1))\sum_{j=1}^N z_j\frac{\partial}{\partial z_j}
-2k\sum_{j=1}^N \sum_{l\ne j}
\frac{ z_j^2}{z_j-z_l}\frac{\partial}{\partial z_j}\,,\notag
\end{align}
   acting on permutation invariant functions from
  $C^2(\mathbb T^N)$. The operator  $H_k$  appears in a prominent way in the particle models of Calogero-Sutherland type 
  on $\mathbb T$; see Section 2 of [LV]. It is obtained from the Calogero-Sutherland Hamiltonian by conjugation with the  ground state. 
  The operator $ H_k$ is the generator of the Feller diffusions $Z_k:=(Z_{t,k}:=e^{iX_{t,k}})_{t\ge0}$
  on the alcove $$\mathbb A_N:=\{e^{ix}: \> x\in C_N\}\subset \mathbb T^N.$$
  Further,
  the operators $\widetilde  H_k:=\frac{1}{k} H_k$
  generate the diffusions $\widetilde Z_k:=(\widetilde Z_{t,k}:=e^{i\widetilde X_{t,k}})_{t\ge0}$
  for $k\in]0,\infty[$. Clearly, this  also works for $k=\infty$ where 
     $\widetilde Z_\infty$  is deterministic  for deterministic initial data.

\begin{remark}  Besides the generators in  (\ref{Lap-Vinet-op}), also the operators
  \begin{equation}
    D_k:= \sum_{j=1}^N z_j^2 \frac{\partial^2}{\partial z_j^2}     +2k\sum_{j=1}^N \sum_{l\ne j}
    \frac{ z_j^2}{z_j-z_l}\frac{\partial}{\partial z_j}
    \end{equation}
appear in the literature; see e.g.~\cite{F, St, OO}.
Here
$-H_k=D_k+E_k\,$ with the Euler operator 
\begin{equation}E_k:=(1-k(N-1))\sum_{j=1}^N z_j\frac{\partial}{\partial z_j}\end{equation}
which commutes with $D_k$. The $-D_k$ are  generators of   diffusions
     with  additional  drift on $\mathbb T$ which
      rotates  the complete system  at some fixed  speed.
     Clearly,  the subsequent results can be easily translated to $-D_k$.
From a stochastic point of view, the diffusions associated with the operators $H_k$ seem to be the most natural ones.
\end{remark}

     We  recall that the (symmetric) eigenfunctions of  $H_k$ are Jack polynomials. To become precise, we
introduce the following notations.
We write
$$\Lambda_N^+ =\{ \lambda\in \mathbb Z_+^N: \lambda_1 \geq \cdots \geq \lambda_N\}$$ 
for the set of partitions of length at most $N$. Denote  
by $C_\lambda^\alpha\,,\, \lambda \in \Lambda_N^+,$  the Jack polynomials of index $\alpha>0$ in $N$ variables with the normalization
\begin{equation}\label{normal}
(z_1 + \cdots + z_N)^m = \sum_{|\lambda|=m} C_\lambda^\alpha(z)\quad (m \in \mathbb Z_+);
\end{equation}
see
\cite{St, BF}. The  $C_\lambda^\alpha$ are symmetric  and homogeneous of degree
$|\lambda|:= \lambda_1 + \cdots + \lambda_N$. Moreover, by \cite{St, BF}, the 
$C_\lambda^\alpha$ with index $\alpha = 1/k$ are eigenfunctions of $D_k$
with  eigenvalues
$$d_\lambda(k)= \sum_{j=1}^N \lambda_j(\lambda_j-1+2k(N-j)).$$
In addition, as  $C_\lambda^\alpha$ is homogeneous of degree $|\lambda|$,
$$E_kC_\lambda^\alpha = (1-k(N-1))|\lambda|\,C_\lambda^\alpha.$$ In summary we obtain:

\begin{lemma}\label{jack-eigenfunctions} For $\lambda \in \Lambda_n^+$, $k\in]0,\infty[$ and $\alpha=1/k$,
     $ C_\lambda^\alpha  $ is an eigenfunction of $\widetilde  H_k = \frac{1}{k}H_k$ with eigenvalue
     $$-\sum_{j=1}^N \lambda_j\bigl(\frac{\lambda_j}{k} +N+1-2j\bigr)\,\le0.$$
  \end{lemma}

It is well-known (c.f. \cite{HO}) that the polynomials  $C_\lambda^{1/k}, \, \lambda \in \Lambda_n^+\, $
form a complete orthogonal system of
$L^2(\mathbb T^N, \mu_k)^W$  with the probability measure  
\begin{equation}\label{stat-dist-torus}
  d\mu_k(z) = \phi_k(z)dz,\,\,\, \phi_k(z):=
  \,c_k\cdot \!\prod_{j,l: l\ne j} |z_j-z_l|^k\cdot {\bf 1}_{\mathbb A_N}(z),\end{equation}
where $c_k>0$ is a normalization constant and $dz$ denotes the Haar measure on $\mathbb T^N$.
Notice that these measures appear also in the context of circular $\beta$-ensembles in random matrix theory.

We also notice that the elementary symmetric polynomials $e_l$ ($l=0,\ldots,N$) in $N$ variables, which are determined by
$$ \prod_{j=1}^N (y- x_j) \, = \sum_{l=0}^N (-1)^l e_l(x)\, y^{N-l}  \quad (y \in \mathbb C),$$ 
are Jack polynomials for all  $k>0$ up to normalization.
More precisely,
by  (\ref{normal}), 
\begin{equation}\label{elem-symm-a} e_l(x)=\frac{1}{l!}  C_{\lambda}^\alpha(x)\>\>\text{with}\>\>
\lambda= 1^l= (\underbrace{1,\ldots,1}_{l \text{ times}},0,\ldots,0) \quad(l=0,\ldots,N).
\end{equation}
Thus in view of Lemma \ref{jack-eigenfunctions}, the polynomials  $e_l$ are eigenfunctions of 
 $\widetilde  H_k$ with eigenvalues $-l\bigl(\frac{1}{k} +N-l\bigr)$. 
Moreover,  by a continuity argument or by direct computation, 
the  $e_l$ are also eigenfunctions of
$$\widetilde  H_\infty= (N-1)\sum_{j=1}^N z_j\frac{\partial}{\partial z_j}-2\sum_{j=1}^N \sum_{l\ne j}
\frac{ z_j^2}{z_j-z_l}\frac{\partial}{\partial z_j}  $$
with  eigenvalues  $-l(N-l).$

We now use these properties to construct martingales from our processes
 $\widetilde Z_k$ for $k\in]0,\infty[$. We recall that by Dynkin's formula (see e.g. Section III.10 of \cite{RW}), the following holds for the generator $L$ of a Feller semigroup and an arbitrary associated Feller process $(X_t)_{t\geq 0}$: if $f$ is a bounded eigenfunction of $L$ with eigenvalue $r\in \mathbb R$, then the process
       $(e^{-rt}f(X_t))_{t\ge0}$ is a martingale w.r.t.~the canonical filtration. We thus have:

     \begin{corollary}\label{cor1-compact-a} For $k\in]0,\infty[$ and $z\in \mathbb A_N$ consider the diffusion
     $(\widetilde Z_{t,k})_{t\ge0}$ on  $\mathbb A_N$ with start in $z$. Then, for $l=0,1,\ldots,N$, the process
     $(e^{l(1/k+N-l)t} e_l(\widetilde Z_{t,k}))_{t\ge0}$ is a martingale. In particular, for  $t\ge0$,
     $$\mathbb E( e_l(\widetilde Z_{t,k}))=e^{-l(1/k+N-l)t}\, e_l(z).$$
     \end{corollary}

     This statement also holds in the deterministic case $k=\infty$, where we also obtain additional information for $t\to \infty$:

     \begin{corollary}\label{cor2-compact-a}
       For each starting point $z$ in the interior of $ \mathbb A_N$, the deterministic process
       $(\widetilde Z_{t,\infty})_{t\ge0}$  satisfies
\begin{equation}\label{determ1-compact-a}
  e_l(\widetilde Z_{t,\infty})= e^{-l(N-l)t}e_l(z) \quad\quad(l=0,1,\ldots,N).
  \end{equation}
Moreover, the limit $Z:=\lim_{t\to\infty}\widetilde Z_{t,\infty}\in\mathbb T^N$ exists and is given by

\begin{equation}\label{determ2-compact-a}Z=(Z_1, Z_1e^{2\pi i /N},\,\ldots, Z_1e^{2\pi i(N-1)/N}) \end{equation}
where $Z_1\in\mathbb T$ is  as follows: If $z=(e^{ix_1},\ldots,e^{ix_N})$ with $(x_1,\ldots,x_N)\in C_N$, then
\begin{equation}\label{determ3-compact-a}
  Z_1= e^{ix_0} \quad\text{with}\quad x_0=\frac{x_1+\ldots+x_N-\pi(N-1)}{N}.
 \end{equation}
     \end{corollary}

     \begin{proof} Eq.~(\ref{determ1-compact-a}) is obvious. As each $\zeta \in \mathbb A_N$ is uniquely determined by the elementary symmetric functions $e_l(\zeta)$, (\ref{determ1-compact-a})     
     and a continuity argument imply that
       $Z=(Z_1,\ldots,Z_N):=\lim_{t\to\infty}\widetilde Z_{t,\infty}\in\mathbb A_N$ exists, and that 
       $$\prod_{j=1}^N (y-Z_j)= y^N - e_N(Z) = y^N-z_1z_2\cdots z_N\,.$$
 This shows that $Z$ has the form as stated in (\ref{determ2-compact-a}) for some $Z_1\in\mathbb T$ with $Z_1^N = z_1 \cdots z_n\,.$ 
 To identify $Z_1$, we write the initial condition as $z=(e^{ix_1},\ldots,e^{ix_N})$ with $(x_1,\ldots,x_N)\in C_N$.
 In the $x$-coordinates, our process  $(\widetilde X_{t, \infty}=(x_{t,1},\ldots,x_{t,N}))_{t\ge0}$
 satisfies the ODE (\ref{dgl-compact-a}). This ODE yields that $x_{t,1}+\ldots+x_{t,N}$ is independent of $t\in[0,\infty]$.
 Therefore, the form of $Z=(e^{ix_{\infty,1}},\ldots,e^{ix_{\infty,N}})$ in (\ref{determ2-compact-a}) yields
 $$x_{1}+\ldots+x_{N}=x_{\infty,1}+\ldots+x_{\infty,N}=\sum_{j=0}^{N-1} (x_{\infty,1}+ j\cdot 2\pi/N)=N x_{\infty,1}+(N-1)\pi.$$
   This implies (\ref{determ3-compact-a}).
\end{proof}

     In the next step we use the decomposition of the processes
     $\widetilde X_{k}=\widetilde X_{k}^{\mathit{diff}}+\widetilde X_{k}^{\mathit{cg}}.$
    Eq.~(\ref{sde-diagonal-compact-a}) and  the expectations of the geometric Brownian motion imply
     \begin{equation}\label{expect-a-cg}
       \mathbb E(e^{-il\widetilde X_{t,k,j}^{\mathit{cg}}})= \mathbb E\left(e^{-\frac{il\sqrt 2}{\sqrt{Nk}}B_t-il\widetilde X_{0,k,j}^{\mathit{cg}}}\right)
      = e^{-\frac{tl^2}{Nk}-il\widetilde X_{0,k,j}^{\mathit{cg}}} \quad(j,l=1,\ldots,N).\end{equation}
      This yields:

     \begin{corollary}\label{cor3-compact-a}
Let  $k\in\,]0,\infty[$ and $x\in C_N$ with $x_1+\ldots +x_N=0.$ Let $z= e^{ix}\in \mathbb A_N$,   
and consider the diffusion
     $(\widetilde Z_{t,k})_{t\ge0}$ on  $\mathbb A_N$ with start in $z$. Then,  for $t\ge0$ and
     $l=1,\ldots,N$,
     $$ \mathbb E( e_l(e^{i\widetilde X_{t,k}^{\mathit{diff}}}))=e^{-l(N-l+(N+l)/(Nk))t} e_l(z).$$
\end{corollary}
 
     \begin{proof} By our initial conditions and  Eq.~(\ref{sde-diagonal-compact-a}) we have
       $\widetilde X_{t,k,j}^{\mathit{cg}}=\widetilde X_{t,k,1}^{\mathit{cg}}$ for all $t,k$ and $j=1,\ldots,N$.
       Hence, the stochastic independence of $\widetilde X_{k}^{\mathit{diff}},\widetilde X_{k}^{\mathit{cg}}$,
Corollary \ref{cor1-compact-a}, (\ref{expect-a-cg}), and the initial condition imply
\begin{align}
  \mathbb E( e_l(e^{i\widetilde X_{t,k}^{\mathit{diff}}}))&=\mathbb E( e_l(e^{i\widetilde X_{t,k}-i\widetilde X_{t,k}^{\mathit{cg}}}))
  =\mathbb E(e_l(\widetilde Z_{t,k})\cdot e^{-il\widetilde X_{t,k,1}^{\mathit{cg}}})\notag\\
  &= \mathbb E(e_l(\widetilde Z_{t,k}))\cdot \mathbb E( e^{-il\widetilde X_{t,k,1}^{\mathit{cg}}})
= e^{-l(1/k+N-l)t}\, e_l(z)\cdot e^{-\frac{tl^2}{Nk}}\notag
\end{align}
as claimed.
\end{proof}

     \begin{example}\label{example-a-central}    For the starting configuration $x=(0,2\pi/N,\ldots,(N-1)2\pi/N)$ and $z=e^x$,
       we have $e_1(z)=\ldots=e_{N-1}(z)=0$ and
     $e_N(z)=(-1)^{N-1}$.  Hence, by  Corollary \ref{cor3-compact-a},
     $$\mathbb E\bigl( e_l(e^{i\widetilde X_{t,k}^{\mathit{diff}}})\bigr)=0    \quad(l=1,\ldots,N-1), \quad\text{and}\quad
    \mathbb E\bigl( e_N(e^{i\widetilde X_{t,k}^{\mathit{\mathit{diff}}}})\bigr)=(-1)^{N-1}e^{-2Nt/k}$$
    for all $t\ge0$.  As $e_0=1$, we conclude that in this case for all $y\in\mathbb C$ and $t\ge0$,
\begin{equation}\label{det-form-dynamic-compact-a}
  \mathbb E\bigl(\prod_{j=1}^{N} (y-e^{i\widetilde X_{t,k,j}^{\mathit{diff}}}) \bigr)=
\mathbb E\bigl(\sum_{l=0}^N y^{N-l} (-1)^l e_l(e^{i\widetilde X_{t,k}^{\mathit{diff}}})\bigl)
=y^N-e^{-2Nt/k}.
\end{equation}
We point out that this result differs from the case $k=\infty$ where $x$ is a stationary solution of the associated ODE by Lemma
\ref{stat-solution-a},
and thus  $\mathbb E\bigl(\prod_{j=1}^{N} (y-e^{i\widetilde X_{t,\infty,j}^{\mathit{diff}}}) \bigr)=y^N-1$ is independent from $t$.
 \end{example}

 We next study the limit $t\to\infty$ for $k\in]0,\infty[$ similar to the limit results
     for $k=\infty$ in Corollary \ref{cor2-compact-a}. We need the following well-known result, which follows
     easily for instance from the explicit formulas for the densities of our diffusions in \cite{RR1}:

     \begin{lemma}
       Let  $k\in\,]0,\infty[$. Then for each starting point in $ \mathbb A_N$,
     the process $(\widetilde Z_{t,k})_{t\ge0}$ converges in distribution for $t\to \infty $ to the
     probability measure $\mu_k$ on $\mathbb A_N$ from \eqref{stat-dist-torus}.
     \end{lemma}

     This observation and Corollary \ref{cor1-compact-a} for $t\to\infty$ imply:

     \begin{corollary}\label{cor4-compact-a}
       Let $Z=(Z_1,\ldots,Z_N)$ be an $\mathbb A_N$-valued random variable with the distribution $\mu_k$ of a circular $\beta$-ensemble.
       Then for each $y\in\mathbb C$,
        $$\mathbb E\bigl(\prod_{j=1}^{N} (y-Z_j) \bigr)=y^N.$$
     \end{corollary}

     \begin{proof}
$$\mathbb E\bigl(\prod_{j=1}^{N} (y-Z_j) \bigr)=\mathbb E\bigl(\sum_{l=0}^N y^{N-l} (-1)^l e_l(Z)\bigl)=y^N.$$
     \end{proof}

     A corresponding result can be also stated under the condition that $Z$ only takes values in the alcove
     $$\mathbb A_N^1:=\{z\in \mathbb A_N: z_1\cdots z_N=1\}.$$
     For this consider the process $\widetilde X_{k}^{\mathit{diff}}$ in the
      decomposition  $\widetilde X_{k}=\widetilde X_{k}^{\mathit{diff}}+\widetilde X_{k}^{\mathit{cg}}$ above.
     Then $e^{i\widetilde X_{k}^{\mathit{diff}}}$ is a diffusion on $\mathbb A_N^1$ which converges for $t\to\infty$
     in distribution to the conditional probability measure $\mu_k^1\in M^1(\mathbb A_N^1)$ of  $\mu_k$ under the condition $\mathbb A_N^1$.
     (This is, up to normalization, the measure with the density \eqref{weight-trig}). 
     Eq. (\ref{det-form-dynamic-compact-a}) now leads to:

     \begin{corollary}\label{cor5-compact-a}
       Let $Z$ be an $\mathbb A_N^1$-valued random variable with distribution $\mu_k^1$. Then for each $y\in\mathbb C$,
        $$\mathbb E\bigl(\prod_{j=1}^{N} (y-Z_j) \bigr)=y^N.$$
    \end{corollary}
    
\begin{remark}
       It can be shown and is well known that the measures $\mu_k$ tend for $k\to\infty$ weakly to the probability measure
       $\mu_\infty$ which appears as image of the uniform distribution on $\mathbb T$ under the mapping
       $$\mathbb T\to \mathbb A_N, \quad z\to (z, z\cdot e^{2\pi i/N},\ldots, z\cdot e^{2\pi i(N-1)/N}).$$
       Corollary \ref{cor4-compact-a} and continuity show that a random variable $Z=(Z_1,\ldots,Z_N)$ on $\mathbb A_N$ with
       distribution $\mu_\infty$ also satisfy $\mathbb E\bigl(\prod_{j=1}^{N} (y-Z_j) \bigr)=y^N$.
A corresponding result holds also in the situation of Corollary \ref{cor5-compact-a}.

       Please notice that this differs from
       the situation in the end of Example \ref{example-a-central}
       for $k=\infty$ where we have a purely deterministic situation and also a slightly different result.
\end{remark}

     Corollaries \ref{cor1-compact-a}, \ref{cor3-compact-a}, \ref{cor4-compact-a}, and \ref{cor5-compact-a}
     have applications to  Brownian motions and  uniform probabilities on compact symmetric spaces of type $A.$
     
     For a first example, consider the space $C(U(N))$ of all conjugacy classes of  $U(N)$,
     which can be identified with $\mathbb A_N$ up to  the cyclic group $\mathbb Z_N$, i.e., $C(U(N))\sim \mathbb A_N/\mathbb Z_N$.
     In fact, the conjugacy classes are
     characterized via the ordered spectra of matrices from $U(N)$ where, say, the eigenvalue
     with the smallest nonnegative argument has the first position. 
     On the other hand, elements in $\mathbb A_N$ describe configurations of ordered points  on $\mathbb T$ where the position of the first 
     entry is arbitrary. It is also well known that
the pushforward of the
uniform distribution (i.e., the normalized Haar measure) of $U(N)$ under
the natural projection $U(N)\mapsto C(U(N))\sim  \mathbb A_N/\mathbb Z_N$
agrees with the  pushforward of  $\mu_k$ with $k=1$ under the canonical mapping $\mathbb A_N\mapsto \mathbb A_N/\mathbb Z_N$.
 Corollary  \ref{cor4-compact-a} thus leads to the following.

\begin{corollary}\label{UN-Haar}
  Let $Z$ be a uniformly distributed $U(N)$-valued random variable. Then for each $y\in\mathbb C$,
  $\, \mathbb E\bigl(\det(yI_N-Z)\bigr)=y^N.$
\end{corollary}

The same procedure also works for  $SU(N),$ where the space $C(SU(N))$ of conjugacy classes  corresponds to $\mathbb A_N^1/\mathbb Z_N$, and 
 the pushforward of the
     uniform distribution  corresponds to  $\mu_k^1\in M^1(\mathbb A_N^1)$ with $k=1$. Hence, 
   by  Corollary \ref{cor5-compact-a}:

 \begin{corollary}\label{SUN-Haar}
  Let $Z$ be an $SU(N)$-valued random variable which is uniformly distributed. Then for  $y\in\mathbb C$,
  $\mathbb E\bigl(\det(yI_N-Z)\bigr)=y^N$.
 \end{corollary}

 Corollaries \ref{UN-Haar} and \ref{SUN-Haar} are special cases of well-known general formulas for integrals of polynomials
 on unitary groups w.r.t.~uniform distributions in \cite{CS}. On the other hand, our approach  leads to
 generalizations of these formulas for Brownian motions on $U(N)$ and $SU(N)$.
 
 For $k=1/2$ or $k=2$,  our results above are related  to the
compact symmetric spaces $U(n)/O(n)$ and $U(2n)/Sp(n)$ associated with the root system $A_{N-1}$.

\section{The non-compact  case of type $A_{N-1}$}\label{Sec_4}

In this section we start with the $W$-invariant Heckman-Opdam
Laplacians
  \begin{equation}
    L_kf(x)= \Delta f(x)+ k\sum_{j=1}^N \sum_{l\ne j}
    \coth\Bigl(\frac{x_j-x_l}{2}\Bigr)\frac{\partial}{\partial x_j}f(x)
   \end{equation}
for  $k\in[0,\infty[$ on the Weyl chamber
$$C_N^A:=\{x\in \mathbb R^N:\> x_1\ge x_2\ge\ldots \ge x_N\}$$
    of type $A_{N-1}$. As in Section 3,  $L_k$ is the generator of a Feller
    diffusion $(X_{t,k})_{t\ge0}$  on $C_N^A$ with reflecting boundaries. Again we study
     the renormalized generators
$\widetilde L_k:=\frac{1}{k} L_k$
which degenerate for $k\to\infty$ into  
$$\widetilde L_\infty=\sum_{j=1}^N \sum_{l\ne j}
\coth\Bigl(\frac{x_j-x_l}{2}\Bigr)\frac{\partial}{\partial x_j}.$$
The process $\widetilde X_{k}:=(\widetilde X_{t,k}:= X_{t/k,k})_{t\ge0}$  for $k\in]0,\infty[$ then
solves the SDE
     \begin{equation}\label{sde-noncompact-a}
       d\widetilde X_{t,k,j}=\frac{\sqrt 2}{\sqrt k}dB_{t,j}+\sum_{l\ne j} \coth\Bigl(\frac{\widetilde X_{t,k,j}-\widetilde X_{t,k,l}}{2}\Bigr)dt
       \quad\quad(j=1,\ldots,N)
     \end{equation}
which degenerates for $k=\infty$ into the ODE
  \begin{equation}\label{dgl-noncompact-a}
    \frac{dx_j}{dt}(t)= \sum_{l\ne j}\coth\Bigl(\frac{x_j(t)-x_l(t)}{2}\Bigr) \quad\quad(j=1,\ldots,N).
    \end{equation}
  Again,  for initial data in the interior of the chamber $C_N^A$,
  the solution $(\widetilde X_{t,\infty})_{t\ge0}$ of these differential equations exists for all
  $t\ge0$ in the interior of $C_N^A$.

  As in the preceding section,  we next decompose the diffusions $\widetilde X_k$ 
     into the  center-of-gravity process $\widetilde X_{k}^{\mathit{cg}}:=(\widetilde X_{t,k}^{\mathit{cg}})_{t\ge0}$ with
     $$\widetilde X_{t,k}^{\mathit{cg}}:= \frac{1}{N}(\widetilde X_{t,k,1}+\ldots+\widetilde X_{t,k,N})\cdot \omega$$
     and the process $\widetilde X_{k}^{\mathit{diff}}:=\widetilde X_{k}-\widetilde X_{k}^{\mathit{cg}}\,$ on the Weyl chamber
     $$C_{N,0}^A:=\{x\in C_N^A: x_1+\ldots+x_N=0\}$$
     which describes the distances of the particles.

  As in the proof of Lemma \ref{independence-a}, the processes $\widetilde X_{k}^{\mathit{diff}}$ and
  $\widetilde X_{k}^{\mathit{cg}}$ are stochastically independent.

In the next step we observe that the processes  $\widetilde X$, $\widetilde X_{k}^{\mathit{diff}}$, $\widetilde X_{k}^{\mathit{cg}}$
admit arbitrary exponential moments for arbitrary deterministic starting points. In fact, for  $\widetilde X_{k}^{\mathit{diff}}$
this follows from Lemma \ref{exponential-moments-ex}. Moreover, $\widetilde X_{k}^{\mathit{cg}}$
is a classical one-dimensional Brownian motion up to scaling and has therefore arbitrary exponential moments.
Finally, the independence of $\widetilde X_{k}^{\mathit{diff}}$ and
  $\widetilde X_{k}^{\mathit{cg}}$ ensures that  $\widetilde X_{k} =\widetilde X_{k}^{\mathit{diff}}+\widetilde X_{k}^{\mathit{cg}}\,$
has this property as well.

With the existence of exponential moments in mind, we now follow  Section 3 and observe that
 the trigonometric 
  elementary symmetric polynomials $$\widetilde e_l(x):=e_l(e^x) \quad (l=0,\ldots,N, \, \, x\in C_N^A)$$
are eigenfunctions of 
  $\widetilde  L_k$ for all $k$ with eigenvalues $l\bigl(\frac{1}{k} +N-l\bigr) \, \geq 0$. 
  This leads to martingales for the diffusions $(\widetilde X_{t,k})_{t\geq 0}$ on $C_N^A,$ similar to Corollaries \ref{cor1-compact-a} and \ref{cor2-compact-a} as follows:

     \begin{corollary}\label{cor1-noncompact-a} For $k\in\,]0,\infty[$ and $x\in C_N^A$ consider the diffusion
     $(\widetilde X_{t,k})_{t\ge0}$ on  $C_N^A$ with start in $x$. Then, for $l=0,1,\ldots,N$, the process
     $\bigl(e^{-l(1/k+N-l)t}\, \widetilde e_l(\widetilde X_{t,k})\bigr)_{t\ge0}$ is a martingale. In particular, for $t\ge0$,
           $$\mathbb E\bigl(\widetilde e_l(\widetilde X_{t,k})\bigr)=e^{l(1/k+N-l)t} \,\widetilde e_l(x).$$
       This result also holds for   $k=\infty$. More precisely, 
       the solution $(\widetilde X_{t,\infty})_{t\ge0}$ of the ODE
       (\ref{dgl-noncompact-a}) with start $x$ in the interior of $C_N^A$ satisfies
       \begin{equation}\label{el-solution-a}
           \widetilde e_l(\widetilde X_{t,\infty})= e^{l(N-l)t}\, \widetilde e_l(x) \quad\quad(l=0,1,\ldots,N).
           \end{equation}
     \end{corollary}

 Using the  independence of  the processes $\widetilde X_{k}^{\mathit{diff}}$ and
  $\widetilde X_{k}^{\mathit{cg}}$, we  also obtain the following analog of
Corollary \ref{cor3-compact-a}.

       \begin{corollary}\label{cor3-a}
         For  $k\in\,]0,\infty[$ and $x\in C_{N,0}^A$  consider the diffusion $\widetilde X_{k}^{\mathit{diff}}$
           on the chamber $C_{N,0}^A$  with start in $x$. Then,  for $t\ge0$ and
     $l=1,\ldots,N$,
         $$ \mathbb E\bigl(  \widetilde e_l(\widetilde X_{t,k}^{\mathit{diff}})\bigr)=e^{l(N-l+ (N+l)/(Nk))t}\,\widetilde e_l(x).$$
           This holds also for  $k=\infty$ where the expectations can be omitted.
     \end{corollary}

Corollary \ref{cor3-a} can be restated as:

           \begin{corollary}\label{cor4-a}
 For  $k\in\,]0,\infty[$ and $x\in C_{N,0}^A$  consider the diffusion $\widetilde X_{k}^{\mathit{diff}}$
           on  $C_{N,0}^A$  with start in $x$. Then,  for $t\ge0$ and $y\in\mathbb C$,
 $$\mathbb E\bigl(\prod_{j=1}^{N} (y-e^{\widetilde X_{k,t,j}^{\mathit{diff}}}) \bigr)=P_{t,N,k,x}(y)$$
           with the polynomial
           $$P_{t,N,k,x}(y):=\sum_{l=0}^N y^{N-l}(-1)^l e^{l(N-l+ (N+l)/(Nk))t} \,\widetilde e_l(x).$$
           \end{corollary}

           For general starting points  $x\in C_{N,0}^A$,
           these  polynomials  do not seem to have particularly nice
           properties.
For the  starting configuration $x=0\in C_{N,0}^A$, which is of particular interest here, we get
           $$P_{t,N,k,0}(y)=\sum_{l=0}^N {N\choose l} (-1)^l y^{N-l} e^{l(N-l+ (N+l)/(Nk))t}.$$
We do not have much information about these polynomials.
 This is a contrast to the Bessel processes of type $A$
where in a corresponding formula  classical Hermite polynomials appear; see \cite{KVW}.

\begin{example}
\begin{enumerate}
\item[\rm{(1)}]  Let $N=2$. We try to solve the ODE (\ref{dgl-noncompact-a}) 
with the singular starting point $x = 0 \in C_2^A.$ In fact, (\ref{dgl-noncompact-a}) yields that $x_1(t)=-x_2(t)$ and
  that $y(t):=x_1(t)-x_2(t)$ satisfies $\frac{dy}{dt}(t)=2 \coth(y(t)/2)$. 
  On the other hand, 
  (\ref{el-solution-a}) for $N=2$, $l=1$ suggests that $$e^{ x_1(t)}+e^{ -x_1(t)}=2e^t$$ and thus
  $x_1(t)={\rm arcosh}(e^t)$. It is now easily checked that in 
  fact,
  $$t \mapsto 
  ({\rm arcosh}(e^t), -{\rm arcosh}(e^t))$$ 
  is continuous on $[0, \infty)$ and solves (\ref{dgl-noncompact-a}) for $t>0.$  Moreover, for $k=\infty$ we have $x^{\mathit{diff}}\!(t)=x(t)$, and 
  $$P_{t,2,\infty,0}(y)=y^2-2e^t \,y +1.$$
  These polynomials have the zeros $e^{\pm x_1(t)}$ as claimed.
  \item[\rm{(2)}] Let $N=3.$ We again try to solve the ODE (\ref{dgl-noncompact-a}) with $x=0\in \partial C_3^A$. Here symmetry arguments imply that the solution of 
    the ODE (\ref{dgl-noncompact-a}) must have the form $x(t)=(x_1(t),0,-x_1(t))$.  On the other hand, formula
    (\ref{el-solution-a}) with $N=2$, $l=1,2$ suggests that  $$e^{ x_1(t)}+e^{ -x_1(t)}+1=3e^{2t}$$ and thus
    $\,x_1(t)={\rm arcosh}\bigl(3(e^{2t}-1)/2\bigr).\,$
    This indeed gives a solution as in example (1).
    Moreover, the polynomial
    $$P_{t,3,\infty,0}(y)=y^3-3e^{2t}y^2+3e^{2t}y-1$$
    has the zeros $1$ and  $e^{\pm x_1(t)}$ as claimed.
\end{enumerate}
\end{example}

Unfortunately, we have no closed formulas for $x(t)$ for general dimension $N$ even for the starting point $x=0\in C_N^A$.

\smallskip
We finally mention that Corollaries  \ref{cor3-a} and  \ref{cor4-a} for $k=1/2,1,2$
have applications to  Brownian motions  on the noncompact symmetric spaces of type A, i.e. on
$GL(N,\mathbb F)/U(N,\mathbb F)$ for $\mathbb F=\mathbb R, \mathbb C, \mathbb H$ similar to the results in the end of Section 3.




\section{The non-compact case of type $BC_{N}$}\label{Sec_5}

We here start with  the nonreduced root system
 $$ R= BC_N = \{\pm e_i, \pm 2 e_i,  \pm(e_i \pm e_j); \>\> 1\leq i < j \leq N\} \subset \mathbb R^N $$
for $N \geq 2$ with the  multiplicity $k=(k_1,k_2, k_3),$ where $k_1, k_2\ge0$, $k_3>0$
are the values on the roots $e_i, 2e_i, e_i \pm e_j$.
As indicated in the introduction, we now reparametrize  the  multiplicity $k$. This will be
not natural at a first glance, but it will turn out to be useful in the end. We 
here mainly follow the notations in \cite{Dem, V} and define
\begin{equation}\label{parameter-change-k-p}
\kappa=k_3, \quad q= N-1+\frac{1+2k_1+2k_2}{2k_3}, \quad p= N-1+\frac{1+2k_2}{2k_3}.
\end{equation}
Then $\kappa > 0$, $q\ge p\ge N-1+1/{2\kappa}$, and
\begin{equation}\label{parameter-change-k-p-rev}k=(k_1,k_2, k_3)= \kappa\cdot\Bigl( q-p, k_{0,2}, 1\Bigr) \quad\text{with}\quad
k_{0,2}:=p-(N-1)-\frac{1}{2\kappa}.\end{equation}
We now regard  $p,q$ as fixed parameters and $\kappa>0$ as a varying parameter, and denote the multiplicity $k$ in 
(\ref{parameter-change-k-p-rev}) by $k_\kappa$. The associated 
$W$-invariant Heckman-Opdam Laplacian \eqref{HO-Lap} is
\begin{align}
  L_\kappa = \Delta +\sum_{i=1}^N \kappa\Bigl((&q-p) \coth(\frac{x_i}{2})+2k_{0,2} \coth(x_i) \, \notag\\
&+\sum_{j: j\ne i}\bigl(\coth(\frac{x_i-x_j}{2})+\coth(\frac{x_i+x_j}{2})\bigr)\Bigr)\partial_{i}
   \end{align}
 on the Weyl chamber
  $$C_N^B:=\{x\in \mathbb R^N:\> x_1\ge x_2\ge\ldots \ge x_N\ge0\}$$
    of type $B_{N}$, associated with $ R_+ = \{e_i, \,2e_i, \,e_i - e_j: 1 \leq i < j \leq N\}.$ 
    As in the preceding sections, the  operators $L_\kappa$ are
    the generators of diffusions $(X_{t,\kappa})_{t\ge0}$  on $C_N^B$ where the paths are reflected at $\partial C_N^B$.
    The renormalized operators $$\widetilde L_\kappa:=\frac{1}{\kappa}L_\kappa\,$$
    then are  the generators of the diffusions  $(\widetilde X_{t,\kappa}:=X_{t/\kappa,\kappa})_{t\ge0}$
    which may be regarded as solutions of the SDE
\begin{align}\label{sde-noncompact-bc}
  d\widetilde X_{t,\kappa,j}=&\frac{\sqrt 2}{\sqrt \kappa}dB_{t,j}+
  \sum_{l\ne j}\Bigl( \coth\Bigl(\frac{\widetilde X_{t,\kappa,j}-\widetilde X_{t,\kappa,l}}{2}\Bigr)+
\coth\Bigl(\frac{\widetilde X_{t,\kappa,j}+\widetilde X_{t,\kappa,l}}{2}\Bigr)
\Bigr)dt \notag\\
&+ \Bigl((q-p)\coth\bigl(\frac{\widetilde X_{t,\kappa,j}}{2}\bigr)+2k_{0,2} \coth(\widetilde X_{t,\kappa,j})\Bigr)dt
       \quad\quad(j=1,\ldots,N).
     \end{align}
For $\kappa\to\infty$, the generator degenerates into
\begin{align}
  \widetilde  L_\infty =& \sum_{i=1}^N \Bigl((q-p) \coth(\frac{x_i}{2})+2 (p-(N-1))\coth(x_i)\notag\\
   & +\sum_{j: j\ne i}\Bigl(\coth(\frac{x_i-x_j}{2})+\coth(\frac{x_i+x_j}{2})\Bigr)\Bigr)\partial_{i},
   \end{align}
and (\ref{sde-noncompact-bc}) becomes the ODE
  \begin{align}\label{dgl-noncompact-bc}
    \frac{dx_j}{dt}(t)=&
    \sum_{l\ne j}\Bigl(\coth\Bigl(\frac{x_j(t)-x_l(t)}{2}\Bigr)+
\coth\Bigl(\frac{ x_j(t)+ x_l(t)}{2}\Bigr)\Bigr)\\
&+ (q-p)  \coth\bigl(\frac{ x_j(t)}{2}\bigr)+2\bigl(p-(N-1)\bigr) \coth x_j(t)
\quad\quad(j=1,\ldots,N).\notag
    \end{align}
  Again,  for initial data in the interior of the chamber $C_N^B$,
  the solution $(\widetilde X_{t,\infty})_{t\ge0}$ of these differential equations exists for
   $t\ge0$ in the interior of $C_N^B$.

  We now turn to eigenfunctions of the operators $\widetilde L_\kappa$ and consider the Heckman-Opdam hypergeometric functions
  $F_{BC}(\lambda,k_\kappa;x)$ associated with  $ BC_N$ in the variable
  $x\in\mathbb R^N,$ with spectral parameter $\lambda\in\mathbb C^N$ and multiplicity $k_\kappa$. 
  By Section 2, for $\kappa\in\,]0,\infty[$,
     the functions $x\mapsto F_{BC}(\lambda,k_\kappa;x)$  are eigenfunctions of the renormalized
     Laplacian $\widetilde L_\kappa$ with the eigenvalues
     $$r_{\lambda,\kappa}:= \frac{1}{\kappa}  \Bigl(\sum_{j=1}^N \lambda_j^2-|\rho(\kappa)|^2\Bigr),$$
     where
     \begin{equation}\label{rho-component}
       \rho(\kappa) = \frac{1}{2} \sum_{\alpha \in R_+} k_\kappa(\alpha) \alpha
     \quad\text{with}\quad
     \rho(\kappa)_j=\frac{\kappa}{2}\bigl((q-p)+2k_{0,2}+2(N-j)\bigr).
     \end{equation}

We further consider the associated (normalized) Heckman-Opdam polynomials  $R_\lambda( k,\,.\,)=R_\lambda^{BC}( k,\,.\,), $
as introduced in (\ref{relation-pol-hyper}), which are indexed by 
$ P_+ = \Lambda_N^+$ (the set of partitions of length at most $N$). These are multivariate generalizations of the classical Jacobi polynomials which are well-studied in the literature, see for instance \cite{BO, L, RR2}.  
 The  polynomials $\widetilde R_\lambda$ defined by
 $$ \widetilde R_\lambda(\cos x):= R_\lambda(k;x)  $$
  form an  orthogonal basis of
$L^2(\mathbb A_N, w_k)$ on 
$$ \mathbb A_N:=\{x\in\mathbb R^N| \> -1\leq y_1\leq ...\leq y_N\leq 1\}$$
with the weight function 
\begin{equation}\label{weight-general-b} w_k(y)
 := \prod_{i=1}^N (1-y_i)^{k_1+k_2-1/2}(1+y_i)^{k_2-1/2} \cdot \prod_{i<j}|y_i-y_j|^{2k_3}.
\end{equation}
Here we are mainly interested in the fact that for $k=k_\kappa$ and $\lambda \in \Lambda_N^+$,  the exponential polynomials
$$H_\lambda(x):=\widetilde R_\lambda(\cosh x)= F_{BC}(\lambda+\rho(k),k;x)$$
are eigenfunctions of $\widetilde L_\kappa$ with the eigenvalues
\begin{align}\label{eigenvalues-bc1}
   r_\lambda&= \frac{1}{\kappa}\langle \lambda,\lambda+2\rho(k)\rangle\notag\\
   &=\frac{1}{\kappa}\sum_{j=1}^N \lambda_j(\lambda_j+k_1+2k_2+2k_3(N-j))\notag\\
   &=\sum_{j=1}^N\lambda_j\Bigl(\frac{\lambda_j-1}{\kappa}+p+q+2-2j\Bigr).
\end{align}
We now consider  the partitions 
 $\lambda(n):= 1^n \in \Lambda_+^N$ for $n=0,\ldots, N.$ It is known (see Section 5 of \cite{V} and in particular Lemma 5.1 there)
that the Jacobi polynomials $ \widetilde R_{\lambda(n)}$ are of the form
\begin{equation}\label{R-lincomb-e}
    \widetilde R_{\lambda(n)}= \sum_{l=0}^n c_{n,l}(p,q) \cdot e_l \quad\text{with } \>c_{n,n}(p,q)\ne0,
    \end{equation}
where the $e_l$ are again the elementary symmetric polynomials in $N$ variables and  the coefficients $c_{n,l}(p,q)\in \mathbb R$ depend on $p,q$ only and not on $\kappa$.
This observation will be crucial in the following and is the reason for our parametrization of $k$ by $p,q,\kappa$ above.
For more details on the $c_{n,l}(p,q)$
we refer to \cite{V}.
In summary, the functions $H_{\lambda(n)}$ with $n = 0, \ldots, N$ are
independent of $\kappa$ and simultanous eigenfunctions
of the operators $\widetilde L_\kappa$ for all $\kappa\in\, ]0,\infty[$ with the eigenvalues
     $$r_n= n(p+q-n+1).$$
Clearly, this observation also holds for $\kappa=\infty$. This implies:

\begin{lemma}\label{martingal-jacobi}
For each  starting point $x\in C_N^B$ of the processes $(\widetilde X_{t,\kappa})_{t\ge0}$ with $\kappa\in\,]0,\infty]$, the 
     processes
\begin{equation}\label{mart-formel}
\Bigl( e^{-r_nt}\cdot H_{\lambda(n)}\bigl(\widetilde X_{t,\kappa}\bigr)\Bigr)_{t\ge0}
\end{equation}
are martingales for  $n =0,\ldots,N$, where the numbers $r_n$ and the functions $H_{\lambda(n)}$ do not depend on 
$\kappa$.

In particular, for $x$ in the interior of $C_N^B$, the solution  $(\widetilde X_{t,\infty})_{t\ge0}$
 of the ODE (\ref{dgl-noncompact-bc}) with $\widetilde X_{0,\infty}=x$ satisfies
$$H_{\lambda(n)}(\widetilde X_{t,\infty})=e^{r_nt}H_{\lambda(n)}(x )\quad\text{for}\quad t\ge0, \> n =0,\ldots,N.$$
\end{lemma}

\begin{proof}
This follows  from our preceding considerations and the fact that   the random variables $H_{\lambda(n)}(\widetilde X_{t,\kappa})$
are integrable for $\kappa<\infty$ and $t>0$ by Lemma \ref{exponential-moments-ex}. 
\end{proof}

We may invert (\ref{R-lincomb-e}) and write  the elementary symmetric polynomials $e_l$ as linear combinations of the  $\widetilde R_{\lambda(n)}$
for $l,n=0,\ldots,N$ with coefficients  independent of $\kappa$. Lemma \ref{martingal-jacobi} thus implies:

\begin{corollary}\label{constant-expectation-general}
  Fix some  deterministic starting point $x\in C_N^B$ as well as the parameters $p,q$. Consider the associated diffusions
   $(\widetilde X_{t,\kappa})_{t\ge0}$ for $\kappa\in\,]0,\infty[$. Then there are coefficients
 $a_{n,l}\in\mathbb R$ for $0\le l\le n\le N$ such that
$$\mathbb E\bigl(e_n(\cosh(\widetilde X_{t,\kappa})) \bigr)=\sum_{l=0}^n a_{n,l}\, e^{r_lt}$$
with $r_0=0$ where the coefficients $a_{n,l}$ and the exponents $r_l$ depend on $p,q$ and $x$ only and not on
  $\kappa$. The same holds for $\kappa=\infty$ and starting points $x$ in the interior of $C_N^B.$ 
\end{corollary}

\begin{example} For $x=0\in C_N^B$, we have $H_{\lambda(n)}(0 )=1\,$ and thus
  $$\mathbb E\bigl(H_{\lambda(n)}(\widetilde X_{t,\kappa})\bigr)=e^{r_nt} \quad\quad(n=0,\ldots,N, \> t\ge0, \,\kappa\in\, ]0,\infty[ ).$$
\end{example}

We finally turn to an application concerning a 
 determinantal formula. Again we  fix some starting point  $x\in C_N^B$ as well as $p,q$ and consider 
the associated diffusions $(\widetilde X_{t,\kappa})_{t\ge0}$ for $\kappa\in\,]0,\infty[$.
 Then by  Corollary \ref{constant-expectation-general}, for all $t\ge0$ and $y\in\mathbb C$,
\begin{align}\label{det-form-bc-noncompact}
  \mathbb E\Bigl(\prod_{j=1}^N \bigl(y- \cosh(\widetilde X_{t,\kappa,j})\bigr)\Bigr) 
&=
    \sum_{n=0}^N (-1)^n\,\mathbb E\bigl(e_n(\cosh(\widetilde X_{t,\kappa})\bigr)\cdot y^{N-n}\\
&=
    \sum_{n=0}^N (-1)^n\,e_n\bigl(\cosh(\widetilde X_{t,\infty})\bigr)\cdot y^{N-n}\notag\\
&=\prod_{j=1}^N \,\bigl(y- \cosh(\widetilde X_{t,\infty,j})\bigr)=: D_{t,x}(y).
\notag\end{align}
It is an interesting task to find particularly nice starting points $x\in  C_N^B$ for which 
$D_{t,x}(y)$ can be determined explicitly.

For instance, in the setting of multivariate Bessel processes of types $A_{N-1}$ or  $B_N$ and start in the origin, 
$D_{t,0}(y)$ is a classical one-dimensional Hermite or Laguerre polynomial of degree $N$ in $y$, which is scaled by a factor $\sqrt t$;
for the details see \cite{KVW}.
Moreover, in the setting of Heckman-Opdam Jacobi processes of type $BC$ on the compact alcove $\mathbb A_N$,
and with the same paramatrization of the multiplicity $k$ by $p,q,\kappa$ as here, there is a (unique) stationary solution $x_0$ of the ODE 
in the interior of $\mathbb A_N$, whose coordinates are the ordered zeroes of some  classical one-dimensional Jacobi polynomial  of degree  $N$ in $y$. The indices of this Jacobi polynomial are determined by $p,q$. This means that for this particular starting point  $x_0$, the function
$D_{t,x_0}(y)$ is just this specific Jacobi polynomial and is independent of $t\ge0$ (due to stationarity). We refer to \cite{V} for further details.

We expect that in our non-compact $BC$ setting, $D_{t,x}(y)$ should be of particular interest when the
associated ODE (\ref{dgl-noncompact-bc}) starts in $x=0\in \partial C_N^B,$ where we 
expect that the corresponding initial value problem is uniquely solvable, as in the Dunkl setting in \cite{VW}.
It seems that the explicit solution of this initial value problem
is more involved than in the cases considered in \cite{KVW, V}.




\bibliographystyle{amsalpha}





\end{document}